\documentclass[12pt, reqno]{amsart}
\usepackage{amsmath}
\usepackage{amssymb}

\numberwithin{equation}{section}

\begin{document}

{\theoremstyle{plain}
    \newtheorem{theorem}{\bf Theorem}[section]
    \newtheorem{proposition}[theorem]{\bf Proposition}
    \newtheorem{claim}[theorem]{\bf Claim}
    \newtheorem{lemma}[theorem]{\bf Lemma}
    \newtheorem{corollary}[theorem]{\bf Corollary}
}
{\theoremstyle{remark}
    \newtheorem{remark}[theorem]{\bf Remark}
    \newtheorem{example}[theorem]{\bf Example}
}
{\theoremstyle{definition}
    \newtheorem{defn}[theorem]{\bf Definition}
    \newtheorem{question}[theorem]{\bf Question}
    \newtheorem{conjecture}[theorem]{\bf Conjecture}
}

\def\move-in{\parshape=1.45true in 5true in}

\def\reg{\operatorname{reg}}

\def\I{\mathcal{I}}
\def\L{\mathcal{L}}
\def\O{\mathcal{O}}
\def\N{\mathbb{N}}
\def\PP{\mathbb{P}}
\def\V{\mathbb{V}}
\def\X{\mathbb{X}}
\def\Y{\mathbb{Y}}
\def\Z{\mathbb{Z}}


\title{Hilbert functions of double point schemes in $\PP^2$}

\author{A.V. Geramita}
\address{Queen's University \\ Department of Mathematics \\ Kingston, Ontario, Canada \\ and Universit\`a degli Studi di Genova \\ Dipartimento di Matematica \\ Genova, Italia}
\email{anthony.geramita@gmail.com}
\urladdr{http://www.mast.queensu.ca/$\sim$tony/}

\author{Huy T\`ai H\`a}
\address{Tulane University \\ Department of Mathematics \\
6823 St. Charles Ave. \\ New Orleans, LA 70118, USA}
\email{tai@math.tulane.edu}
\urladdr{http://www.math.tulane.edu/$\sim$tai/}
\thanks{The second named author is partially supported by NSA grant H98230-11-1-0165 and Louisiana BOR grant LEQSF(2007-10)-RD-A-30}


\begin{abstract} We study the question of whether there is a minimum Hilbert function for double point schemes whose support is $s$ points with generic Hilbert function. Previous work shows that this question has an affirmative answer for $s \le 9$ and for $s = {d \choose 2}$ (for any $d \in \N$). In this paper, we provide evidence in the case $s = {d \choose 2} + 1$, and give an affirmative answer to the question when $s = 11$.
\end{abstract}

\maketitle


\begin{center}
{\it Dedicated to Professor H\`a Huy Kho\'ai on the occasion of his 65th birthday.}
\end{center}


\section{Introduction}

Let $R = k[x_0, \ldots , x_n] = \bigoplus_{j=0}^\infty R_j$ be the standard graded polynomial ring over an algebraically closed field $k$ of characteristic 0. Let $I = \bigoplus_{j=0}^\infty I_j$ be a homogeneous ideal of $R$.  The quotient ring $A = R/I = \bigoplus_{j=0}^\infty A_j$ (with $A_j = R_j/I_j$) is also a graded ring.  The {\it Hilbert function} of $A$,
$
H_A: \N \longrightarrow \N,
$
is defined by
$$
H_A(t) = \dim_k A_t = \dim_k R_t - \dim_kI_t = {t+n\choose n} - \dim_kI_t .
$$

Let $\X = \{ P_1, \ldots , P_s \} \subset \PP^n$ be a set of $s$ distinct points, and let
$$\wp_i = (L_{i1}, \ldots , L_{in} ) \subset R$$
be the defining ideal of $P_i$ (the $L_{ij}$'s, $1 \leq j \leq n$, are linearly independent linear forms). We often write $P_i \leftrightarrow \wp_i$ to indicate this correspondence between a point in $\PP^n$ and its defining ideal. Let
$$
I_\X = \wp_1 \cap \cdots \cap \wp_s
$$
be the defining ideal of $\X$. It is well known that for all $t \geq s-1$ and $A = R/I_\X$, one has $H_\X(t): = H_A(t) = s$.  This naturally raises the question of what are the possible values for $H_\X(t)$ for $t < s-1$ as $\X$ varies over all possible sets of $s$ distinct points in $\PP^n$. This question has been settled (cf. \cite{GMR}); a complete characterization of all the possible Hilbert functions for $s$ points in $\PP^n$ is known.

The situation gets a lot more complicated when one moves beyond the case of simple points (cf. \cite{BC, CHT, GHM1, GHM2, GMS}). In a relatively recent tour-de-force, J. Alexander and A. Hirschowitz \cite{AH} found all the Hilbert functions for the non-reduced, zero dimensional subschemes of $\PP^n$ defined by ideals $\wp_1^2\cap\cdots\cap \wp_s^2$, in the case that the underlying reduced subscheme is a general set of points in $\PP^n$.  This was the key point in their solution to the long-outstanding Waring Problem for Forms (see \cite{AH0, AH, C, Ger, I, I-K}). Note that knowing the Hilbert function of this non-reduced subscheme of $\PP^n$ is equivalent to knowing the dimension of the space of hypersurfaces of any given degree with singularities at the given general set of points.

This result of Alexander-Hirschowitz was one of the motivations for the authors of \cite{GMS} asking whether it was possible, even in $\PP^2$, to find the Hilbert functions of {\bf all} the non-reduced subschemes defined by ideals of the type
$$
\wp_1^2\cap\cdots\cap \wp_s^2
$$
whose underlying reduced scheme was {\bf any} set of $s$ distinct points.  This has turned out to be a very challenging problem.

This problem (as well as some obvious generalizations) has been taken up in several papers since then with some notable successes if the number of points is small ($\leq 9$) (cf. \cite{GHM1, GHM2}).  In this case, i.e. when $s \leq 9$, the exponents on the prime ideals $\wp_i$ can be any positive integers.  Yet, even with all this work, it is fair to say that we are still very far from a general solution to the problem.

One of the ideas of \cite{GMS} was to divide up the problem in a way that might help see what the possibilities are.  More specifically, they asked: if $\chi(H)$ is the set which consists of all the subsets of $s$ points of $\PP^2$ which share the same Hilbert function $H$ and if $\X = \{P_1, \ldots , P_s \} \in \chi(H)$ with $P_i \leftrightarrow \wp_i \subset k[x,y,z] = R$, what can be said about the Hilbert function of $R/J$ when
$$
J = \wp_1^2 \cap \cdots \cap \wp_s^2? \eqno{(*)}
$$

We call the scheme defined by the ideal $J$ the {\it double point scheme supported on} $\X$.  Indeed, it has now become somewhat standard to refer to the scheme defined by $J$ as $2\X$ (and thus referring to $H_{R/J}$ as $H_{2\X}$), even though that leaves open room for misinterpretation.

The {\it generic} Hilbert function for a set $\X$ of $s$ points in $\PP^2$ is
$$
H_{\X}(t) = \min\ \{ s , {t+2\choose 2} \}.
$$
We shall denote this Hilbert function by $H_{gen,s}$. From completely general considerations, it is easy to see that for $\X$ any set of $s$ distinct points, we have
$$
H_{2\X}(t) \leq \min\{ {t+2\choose 2}, 3s \} .
$$
Moreover, using the Theorem of Alexander and Hirschowitz \cite{AH}, this upper bound is achieved (for $s \neq 2, 5$) for almost every set of points $\X \in \chi(H_{gen, s})$.

This leads to a natural question:

\move-in\noindent {\bf Question:} Is there a lower bound for $H_{2\X}(-)$ for $\X \in \chi(H_{gen,s})$ which is actually achieved?

There seems no apparent reason why this question should have a positive answer, but in \cite{GMS} the authors showed that it did have a positive answer for an infinite family of $s$.  More precisely, they proved the following theorem.

\begin{theorem}[\protect{\cite[Theorem 7.10]{GMS}}] \label{thm.intro1}
Let $s = {d \choose 2}$ and let $C_d = \{P_1, \dots , P_s \}$ be the $s$ points which arise as the intersection of $d$ general lines in $\PP^2$.  Then
\begin{enumerate}
\item $H_{C_d}(t) = \min \{ {t+2\choose 2}, {d\choose 2} \}$ for every $t$, and so $C_d \in \chi(H_{gen, s}).$
\item Let $P_i \leftrightarrow \wp_i \subset k[x,y,z]$ and let $J = \wp_1^2\cap\cdots\cap \wp_s^2$ be the defining ideal of $2C_d$. Then
$$
H_{R/J}(t) = H_{2C_d}(t) =  \left\{
\begin{array}{rcl}
{t+2 \choose 2} & \text{if} & t \le d-1 \\
{d+1 \choose 2} + (t+1-d)d & \text{if} & d \le t \le 2d-3 \\
3{d \choose 2} & \text{if} & t \ge 2d-2.
\end{array} \right.
$$
\item If $\X \in \chi(H_{gen, s})$ then we have
$$
H_{2C_d}(t) \leq H_{2\X}(t) \hbox{ for every $t \in \N$.}
$$
\end{enumerate}
\end{theorem}

Theorem \ref{thm.intro1} shows that there is a minimum Hilbert function for the double point schemes supported on $s = 3, 6, 10, 15, 21, ...$ points with general Hilbert function.  The work done in \cite{GHM1, GHM2} further shows that such minima also exist for $s \leq 9$ points.  Thus, the first open case for the {\bf Question} is for $s = 11$ points in $\PP^2$.

The goal of this paper is to give an affirmative answer to the {\bf Question}  when $s = 11$.  Since our proof involves looking at many separate cases, we hope that our solution will inspire a more interesting attack on this thorny problem.

In the next section, we consider the case when $s$ is of the form $s = {d \choose 2} + 1$. In particular, when $d = 5$, $s = 11$. We recall from \cite{GMS} those configurations of $s$ points over which the double point scheme is conjectured to have the minimum Hilbert function. We provide supportive evidence to the {\bf Question} in this case by proving, in Proposition \ref{pro.atmostt}, that the assertion holds for the ``first half'' of the Hilbert function. The last section is devoted to the case $s = 11$. An affirmative answer to the {\bf Question} in this case is obtained by a series of case analyses, which are to be found in Propositions \ref{pro.60} - \ref{pro.62}.

\noindent{\bf Acknowledgement.} Our work began when the second named author visited the first named author at Queen's University. We would like to thank Queen's University and its Mathematics department for their hospitality. We are very grateful for the many splendid conversations we have had on these questions with Brian Harbourne, Juan Migliore and Mike Roth. We would also like to thank an anonymous referee for a careful read of the paper and useful suggestions.


\section{Line configurations and the first half of Hilbert functions}

In this section, we give some evidence that the {\bf Question} should have a positive answer for $s = {d \choose 2} + 1$ points.

We start by recalling the description of those configurations of points which are conjectured to be the support of double point schemes having the minimum Hilbert functions.

\begin{defn} \label{def.conf}
Let $L_1, \dots, L_{d+1}$ be a set of $(d+1)$ general lines in $\PP^2$.  (Thus, each $L_i$ meets the remaining $d$ lines at $d$ distinct points.) Denote by $C_{d+1}$ the configuration of ${d+1 \choose 2}$ intersection points of the $L_i$s. For any $0 \le r \le d$, let $C_{d,r}$ be the subconfiguration of $C_{d+1}$ obtained by removing $(d-r)$ intersection points that are on the single line $L_{d+1}$.
\end{defn}

By \cite[Lemma 7.8]{GMS}, it is known that $C_{d,r}$ has the Hilbert function of ${d \choose 2}+r$ general points in $\PP^2$. In particular, $C_{d,1}$ has the Hilbert function of $s = {d \choose 2} + 1$ generic points in $\PP^2$. As before, we denote this generic Hilbert function by $H_{gen,s}$.


By abuse of notation, we shall use $F$ to denote both a homogeneous form in the polynomial ring $R = k[x, y, z]$ and the curve $\V(F)$ in $\PP^2$.

\begin{proposition} \label{pro.atmostt}
Assume that $d \ge 3$. If $\X \in \chi(H_{gen, s})$, $s = {d \choose 2} +1$, then $H_{2\X}(t) = {t+2 \choose 2}$ for all $t \le d$.
\end{proposition}

\begin{proof} Let $2\X$ denote the double point scheme whose support is $\X$ and let $J_{2\X}$ be its defining ideal in $R$. It suffices to show that the ideal $J_{2\X}$ contains no form of degree $d$. Suppose, by contradiction, that $F \in J_{2\X}$ is a form of degree $d$. We shall consider two cases.

\noindent{\bf Case 1.} $F$ is reduced. Since an irreducible curve of degree $d$ has at most ${d-1 \choose 2}$ singular points, $F$ must be reducible. Let $F = G_1 \dots G_r$ be the factorization of $F$ into irreducible forms, where $\deg G_i = \lambda_i$ (in particular, $2 \le r \le d$ and $\sum_{i=1}^r \lambda_i = d$). Each curve $G_i$ contains at most ${\lambda_i - 1 \choose 2}$ singular points, and the curves $G_i$ intersect at, at most, $\sum_{i < j} \lambda_i \lambda_j$ distinct points. Thus, $F$ has at most
$$D = \sum_{i=1}^r {\lambda_i - 1 \choose 2} + \sum_{i< j} \lambda_i\lambda_j$$
singular points. However, since $F$ is singular at all the points in $\X$, we must have $D \ge {d \choose 2} + 1 = {d-1 \choose 2} + d$. A contradiction will result if we can show that
\begin{align}
D = {d-1 \choose 2} + r - 1  \label{eq1}
\end{align}
since ${d-1 \choose 2} + r-1 <  {d-1 \choose 2} + d$.

Indeed, we shall use induction on $r$ to prove this equality. For $r = 2$, we can easily verify the equality
$${\lambda_1 - 1 \choose 2} + {\lambda_2 - 1 \choose 2} + \lambda_1\lambda_2 = {\lambda_1 + \lambda_2 - 1 \choose 2} + 1.$$
Suppose that (\ref{eq1}) is proved for $r-1$. That is, suppose
$$\sum_{i=1}^{r-1} {\lambda_i -1 \choose 2} + \sum_{i < j < r} \lambda_i\lambda_j = {d'-1 \choose 2} + (r-2)$$
where $d' = \sum_{i=1}^{r-1} \lambda_i$. It follows from the case when $r = 2$ that
$${d' - 1 \choose 2} + {\lambda_r - 1 \choose 2} + d'\lambda_r = {d'+\lambda_r - 1 \choose 2} + 1.$$
That is, ${d' - 1 \choose 2} + {\lambda_r - 1 \choose 2} + (\sum_{i=1}^{r-1}\lambda_i)\lambda_r = {d - 1 \choose 2} + 1.$ Thus, together with the induction hypothesis, we have
$$\sum_{i=1}^r {\lambda_i - 1 \choose 2} + \sum_{i < j} \lambda_i\lambda_j = {d-1 \choose 2} + 1 + (r-2) = {d-1 \choose 2} + (r-1).$$

\noindent{\bf Case 2.} $F$ is non-reduced. In this case, $F$ can be written as $F = G^2H$, where $G$ and $H$ are homogeneous forms and $H$ is reduced. Let $\deg G = \theta$ and $\deg H = \lambda$ (then $d = 2\theta+\lambda$).

As in Case 1, $H$ has at most ${\lambda \choose 2}$ singular points. This implies that at least ${d \choose 2} + 1 - {\lambda \choose 2}$ of the points of $\X$ come from $G^2$ or the intersection $G \cap H$. I.e., $G$ must contain at least ${d \choose 2} + 1 - {\lambda \choose 2}$ points of $\X$. Notice that ${2\theta+\lambda \choose 2} + 1 - {\lambda \choose 2} > {\theta+2 \choose 2} -1$.

It can be further seen from the Hilbert function of $\X$ and its difference function,
\begin{align*}
H_{\X} & : \ \ 1 \ \ 3 \ \ 6 \ \ \dots \ \ {d \choose 2} \ \ {d \choose 2} + 1 \ \ \rightarrow \\
\Delta H_{\X} & : \ \ 1 \ \ 2 \ \ 3 \ \ \dots \ \ d-1 \ \ 1 \ \ 0 \ \ \rightarrow
\end{align*}
that for $\theta < d$ a degree $\theta$ curve contains at most $(1 + 2 + \dots + \theta + \theta) = {\theta+2 \choose 2} - 1$ points of $\X$. This gives a contradiction.
\end{proof}


\section{First infinitesimal neighborhood of 11 points}

This section is devoted to showing that the {\bf Question} has an affirmative answer for $s = 11$ points. Let $C_{5,1}$ be the configuration defined above, and let $2C_{5,1}$ be the double point scheme supported on $C_{5,1}$. By \cite[p. 607]{GMS}, the Hilbert function of $2C_{5,1}$ and its difference function are
\begin{align*}
H_{2C_{5,1}} & : \ \ 1 \ \ 3 \ \ 6 \ \ 10 \ \ 15 \ \ 21 \ \ 26 \ \ 31 \ \ 32 \ \ 33 \ \ \rightarrow \\
\Delta H_{2C_{5,1}} & : \ \ 1 \ \ 2 \ \ 3 \ \ 4 \ \ 5 \ \ 6 \ \ 5 \ \ 5 \ \ 1 \ \ 1 \ \ 0 \ \ \rightarrow
\end{align*}

We shall show that the Hilbert function $H_{2C_{5,1}}$ is the smallest Hilbert function for double point schemes whose support has the generic Hilbert function $H_{gen, 11}$.

For the rest of the paper, let $\X \in \chi(H_{gen,11})$. Our goal is to prove that $H_{2C_{5,1}} \le H_{2\X}$, i.e.,
$$
H_{2C_{5,1}}(t) \le H_{2\X}(t) \text{ for every } t \in \N. \eqno{(**)}
$$

It follows from Proposition \ref{pro.atmostt} that
$$H_{2C_{5,1}}(t) = H_{2\X}(t) \text{ for any } t \le 5.$$
Thus, it remains to prove the inequality $(**)$ for $t \ge 6$. Clearly, $H_{2\X}(6) \le \dim_k R_6 = 28$. Therefore, we shall proceed by considering different cases when $H_{2\X}(6)$ is 28, 27, 26 or $\le 26$.

The following lemma appears to be known, but we could not find it in the literature.

\begin{lemma} \label{BertiniType}
Let $\L$ be a linear series of dimension at least 2 on a projective scheme $X$. Assume that $\L$ has no fixed component. Then a general divisor in $\L$ is reduced.
\end{lemma}

\begin{proof} Let $F$ be a general divisor of $\L$. Suppose that $F$ is not reduced. Then $F$ can be written as $F = G^2H$, where $G$ is irreducible. Observe that all points (infinitely many) of $G$ are singular points of $F$. By Bertini's theorem, these points are in the base locus of $\L$. Let $F'$ be any other element in $\L$. Since all points of $G$ are in the base locus of $F'$, $G$ is a factor of $F'$. Thus, $G$ is a fixed component of $\L$, a contradiction.
\end{proof}

\begin{proposition} \label{pro.60}
If $H_{2\X}(6) = 28$ then $H_{2\X}(7) \ge 31$. This, in particular, implies that $H_{2\X}(8) \ge 32$, $H_{2\X}(9) \ge 33$ and thus $H_{2\X} \ge H_{2C_{5,1}}$.
\end{proposition}

\begin{proof} Since $\reg I_\X = \reg I_{C_{5,1}} = 5$, by a result of \cite{GGP}, we have $\reg I_{2\X} \le 10$. Thus, in this case, the difference function of the Hilbert function of $2\X$ has the form
$$\Delta H_{2\X}\ : \ \ 1 \ \ 2 \ \ 3 \ \ 4 \ \ 5 \ \ 6 \ \ 7 \ \ a \ \ b \ \ c \ \ 0 \ \ \rightarrow$$
where $a+b+c = 5$ and $a \ge b \ge c$. The possibilities for $(a,b,c)$ are $(5,0,0)$, $(4,1,0)$, $(3,2,0)$, $(3,1,1)$ and $(2,2,1)$.

In the first 4 cases, $H_{2\X}(7) \ge 28+3 = 31$. In the last case, by \cite[Proposition 5.2]{GMR}, there exists a conic $C$ containing a subscheme of $2\X$ with multiplicity 18.

A reduced conic has at most 1 singular point, so either $C$ contains 9 simple points of $\X$ or a double point in $2\X$ and some simple points in $\X$. In the first case, there are too many points of $\X$ on a conic, violating the hypothesis that $\X$ has generic Hilbert function. In the second case, the multiplicity of the subscheme of $2\X$ on $C$ is an odd number, which cannot be 18.

If $C$ is non-reduced then $C = L^2$, where $L$ is a linear form. This implies that $L$ contains 9 points of $\X$, which again violates the hypothesis that $\X$ has generic Hilbert function.
\end{proof}

\begin{proposition} \label{pro.61}
If $H_{2\X}(6) = 27$ then $H_{2\X}(7) \ge 31$. Thus $H_{2\X} \ge H_{2C_{5,1}}$.
\end{proposition}

\begin{proof} The hypothesis $H_{2\X}(6) = 27$ implies that the defining ideal $I_{2\X}$ of $2\X$ has a unique generator, say $F_6$, in degree 6. We need to show that $\dim_k \big(I_{2\X}\big)_7 \le 5$.

Consider any homogeneous form $H_7$ of degree 7 in $I_{2\X}$. By B\'ezout's Theorem, $H_7$ and $F_6$ should have an intersection divisor of degree 42, but the 11 double points of $2\X$ are on both $H_7$ and $F_6$, contributing at least 44 to that degree. This can only the be the case if $H_7$ and $F_6$ share a common component. In particular, $H_7$ is reducible. This is true for any form in $\big(I_{2\X}\big)_7$. Thus, by Bertini's theorem either $\big(I_{2\X}\big)_7$ has a fixed component or $\big(I_{2\X}\big)_7$ is composed with a pencil.

\noindent{\bf We first assume that $\big(I_{2\X}\big)_7$ has a fixed component.}

Notice that $F_6x, F_6y, F_6z \in \big(I_{2\X}\big)_7$. Therefore, any fixed component of $\big(I_{2\X}\big)_7$ is a factor of $F_6$. We shall denote this fixed component by $F$. The following (rather long case by case) argument  focuses on the various possibilities for the degree of the fixed component $F$ of $\big(I_{2\X}\big)_7$.

\noindent {\bf Case a:} $\deg F = 6$. In this case $F = F_6$, and clearly elements of $\big(I_{2\X}\big)_7$ are of the form $FL$, where $L$ is a linear form. We then have $\dim_k \big(I_{2\X}\big)_7 = 3.$

\noindent {\bf Case b:} $\deg F = 5$. In this case, $F_6 = FL$ and $H_7 = FC_2$, where $L$ is a linear form and $C_2$ is a quadratic form.

By Proposition \ref{pro.atmostt}, $\big(I_{2\X}\big)_5 = (0)$. Thus, $F$ is not singular at all the points of $\X$.  Since $FL$ is singular at all the points of $\X$ this implies that $L$ must contain at least one point of $\X$.  By the same reasoning, the conic $C_2$ contains that same point. This puts at least one linear condition on the space of such quadrics and hence, the space of such quadrics has dimension at most 5. In particular, $\dim_k \big(I_{2\X}\big)_7 \le 5$.

\noindent {\bf Case c:} $\deg F = 4$. In this case $F_6 = FC_2$ and $H_7 = FC_3$, where $C_2$ is a quadratic form and $C_3$ is a cubic form. Again, since $F$ is not singular at all the points of $\X$, $C_2$ and $C_3$ are determined by  vanishing conditions at a subset $\Y$ of $\X$ (a subset that cannot be on a line; otherwise, $FL$ would be a form of degree 5 singular at all the points of $\X$). The Hilbert function of $\Y$ is, thus, of the form
$$H_\Y \ : \ \ 1 \ \ 3 \ \ 5 \ \ \ge 5 \ \ \cdots$$
This implies that the space of cubics through $\Y$ has dimension at most 5 and hence that $\dim_k \big(I_{2\X}\big)_7 \le 5$.

\noindent {\bf Case d:} $\deg F = 3$. In this case $F_6 = FC_3$ and $H_7 = FQ_4$, where $C_3$ is a cubic form and $Q_4$ is a quartic form. Note that $F \not= C_3$. As in the previous two cases, $C_3$ and $Q_4$ are determined by vanishing conditions imposed by a subscheme $\Z$ of $2\X$ (that is not contained in $F$).

\noindent{\it Claim 1.} The multiplicity of $\Z$ is at least 10.

\noindent{\it Proof of the Claim.} We shall consider different subcases depending on whether $F$ and $C_3$ share a common component.

\noindent{\it Case d.1:} $F$ and $C_3$ share a common component.

Suppose the highest degree common component of $F$ and $C_3$ is a conic, i.e., $F = C_2L$ and $C_3 = C_2L'$, where $C_2$ is a conic, $L$ and $L'$ are lines. Since $C_2$ contains at most 8 points of $\X$, at least 3 double points of $2\X$ must lie on the intersection of $L$ and $L'$. This implies that $L = L'$, i.e., $F = C_3$, a contradiction.

Suppose now that the highest degree common component of $F$ and $C_3$ has degree 1, i.e., $F = LC_2$ and $C_3 = LC_2'$, where $C_2$ and $C_2'$ are quadratic forms. Since the line $L$ contains at most 5 points of $\X$, if either $C_2$ or $C_2'$ has no singular points then for $FC_3$ to contain $2\X$, the intersection of $C_2$ and $C_2'$ must have at least 5 points of $2\X$. This means that $C_2$ and $C_2'$ share a component, and we are back to the previous case. It remains to consider the possibility that both $C_2$ and $C_2'$ have singular points, i.e., that $C_2 = L_1L_2$ and $C_2' = L_1'L_2'$, where $L_i \not= L_j'$. Again, since a line contains at most 5 points of $\X$, we must have $L_1 \not= L_2$ and $L_1' \not= L_2'$. In this case, $\X$ consists of exactly 5 points on $L$, the intersection of $L_1$ and $L_2$, the intersection of $L_1'$ and $L_2'$, and 4 intersection points of $L_i$ and $L_j'$. Therefore, the subscheme $\Z$ determining $Q_4$ consists of a double point at the intersection of $L_1'$ and $L_2'$, and 9 points other than the intersection of $L_1$ and $L_2$. Clearly, the multiplicity of $\Z$ is 12.

\noindent{\it Case d.2:} $F$ and $C_3$ share no common component.

The possibilities for a cubic $F$ are: $F$ is irreducible, $F = LC$, $F = L_1L_2.L_3$, $F = L^2L_1$ and $F = L^3$ (where $L, L_1, L_2$ and $L_3$ denote linear forms, and $C$ denotes an irreducible conic).

If $F$ is an irreducible cubic then $F$ contains at most 1 singular point, and $F$ and $C_3$ intersect at 9 points. Thus, $\Z$ consists of at least a singular point and 9 simple points, whence the multiplicity of $\Z$ is at least 12.

If $F = LC$ then $F$ has 2 double points. Since a cubic $F$ cannot contain all the 11 points of $\X$, $C_3$ has to contain at least a double point of $2\X$. Thus, in this case, the subscheme $\Z$ consists of $a$ double points (for some $a \ge 1$) and at least $(11-2-a)$ simple points in the intersection of $F$ and $C_3$. The multiplicity of $\Z$ is then at least $3a + (11-2-a) = 2a+9 \ge 11$.

If $F = L_1L_2L_3$ then $F$ has 3 double points. As before, $C_3$ must contain a double point of $2\X$, and thus, the subscheme $\Z$ consists of $a$ double points (for some $a \ge 1$) and at least $(11-3-a)$ simple points. The multiplicity of $\Z$ is then at least $3a + (11-3-a) = 2a + 8 \ge 10$.

If $F = L^2L_1$ then $L$ contains at most 5 points of $\X$. If $a$ out of 3 intersection points of $L_1$ and $C_3$ are points in $\X$ then $C_3$ must contain at least $(11-5-a)$ double points of $2\X$. Thus, the subscheme $\Z$ consists of at least $(11-5-a)$ double points and $a$ simple points. In this case, the multiplicity of $\Z$ is at least $3(11-5-a) + a = 18-2a \ge 12$, since $a \le 3$.

If $F = L^3$ then since $L$ contains at most 5 points of $\X$, the subscheme $\Z$ must have at least 6 double points of $2\X$; and thus, $\Z$ has multiplicity at least 18.
The claim is proved. \qed

We shall proceed with the proof of the proposition, having shown that the multiplicity of $\Z$ is at least 10. Notice that since $\Z$ belongs to exactly one cubic ($\dim_k \big(I_{2\X}\big)_6 = 1$), the Hilbert function of $\Z$ has the form
$$H_\Z \ : \ \ 1 \ \ 3 \ \ 6 \ \ 9 \ \ \ge 10 \ \ \cdots $$
This implies that the space of quartics containing $\Z$ has dimension at most 5. In particular, we have $\dim_k \big(I_{2\X}\big)_7 \le 5$.

\noindent {\bf Case e:} $\deg F = 2$. In this case $F_6 = FQ_4$ and $H_7 = FQ_5$, where $Q_4$ and $Q_5$ are quartic and quintic forms. As before, $Q_4$ and $Q_5$ are determined by a subscheme $\Z$ of $2\X$.

Since the conic $F$ contains at most 8 points of $\X$, $Q_4$ must contain at least 3 double points in $2\X$. Moreover, $F$ has at most one double point. Therefore, at least 7 of the intersection points of $F$ and $Q_4$ are points of $\X$. In this case, the subscheme $\Z$ consists of at least 3 double points and 7 simple points, and has multiplicity at least 16. Since there is a unique quartic (namely $Q_4$) that contains $\Z$, the Hilbert function of $\Z$ is of the form
$$H_\Z \ : \ \ 1 \ \ 3 \ \ 6 \ \ 10 \ \ 14 \ \ (15 \text{ or } 16) \ \ \ge 16 \ \ \dots$$

If $H_\Z(5) = 16$ then the space of quintics containing $\Z$ has dimension at most 5. In particular this implies that $\dim_k \big(I_{2\X}\big)_7 \le 5$ and we would be done.

If, however, $H_\Z(5) = 15$ then the Hilbert function of $\Z$ must be
$$H_\Z \ : \ \ 1 \ \ 3 \ \ 6 \ \ 10 \ \ 14 \ \ 15 \ \ 16 \ \ \rightarrow$$
(the increment from $14$ to $15$ forces $H_\Z(6) \le 16$). In this case, difference Hilbert function of $\Z$ is
$$\Delta H_\Z \ : \ \ 1 \ \ 2 \ \ 3 \ \ 4 \ \ 4 \ \ 1 \ \ 1 \ \ 0 \ \ \rightarrow $$
By \cite[Proposition 5.2]{GMR}, there is a line containing 7 points of $\Z$ (hence at least 7 points of $\X$). This contradicts the fact that $\X$ has  generic Hilbert function.

\noindent {\bf Case f:} $\deg F = 1$. In this case $F_6 = FQ_5$ and $H_7 = FS_6$, where $Q_5$ and $S_6$ are quintic and sextic forms, respectively.

The line $F$ contains at most 5 points of $\X$. Thus, $Q_5$ must contain at least 6 double points of $2\X$. Moreover, if $Q_5$ contains at least 10 double points of $2\X$ then, since the space of lines going through the 11th point of $\X$ has dimension 2, the space of such $F_6$'s would have dimension 2 - a contradiction to the hypothesis. Therefore, $Q_5$ contains at most 9 double points of $2\X$. Suppose $\Y$ is the subscheme of $2\X$ determining $Q_5$ (and $S_6$) consisting of $a$ double points (for some $6 \le a \le 9$) and $(11-a)$ simple points on the line $F$.

We shall denote by $\Z$ the subscheme of $a$ double points in $\Y$. Let $I_\Y$ and $I_\Z$ denote the defining ideals of $\Y$ and $\Z$ in $\PP^2$. Also let $\I_\Y$ and $\I_\Z$ denote the ideal sheaves of $\Y$ and $\Z$. We have the following short exact sequence
\begin{align}
0 \rightarrow \I_\Z(t-1) \stackrel{\times F}{\rightarrow} \I_\Y(t) \rightarrow \I_{\Y \cap F}(t) \rightarrow 0. \label{eq2}
\end{align}
Note further that the homogeneous pieces $\big(I_\Y\big)_t$ and $\big(I_\Z\big)_t$ can be identified with $H^0(\I_\Y(t))$ and $H^0(\I_\Z(t))$. It follows from the short exact sequence (\ref{eq2}) that
\begin{align}
h^0(\I_\Y(6)) \le h^0(\I_\Z(5)) + h^0(\I_{\Y \cap F}(6)). \label{eq3}
\end{align}

Observe that $\Y \cap F$ consists of $(11-a)$ simple points on a line, so $\I_{\Y \cap F}(6) \simeq \O_{\PP^1}(6-(11-a)) = \O_{\PP^1}(a-5)$. Thus, $h^0(\I_{\Y \cap F}(6)) = h^0(\O_{\PP^1}(a-5)) = a-4$. Observe further that the {\bf Question} has a positive answer for $\le 9$ points (see \cite{GHM1, GHM2}). Therefore:
\begin{enumerate}
\item If $\Z$ consists of 6 (i.e., $a = 6$) double points then the Hilbert function of $\Z$ is at least
\begin{align*}
1 \ \ 3 \ \ 6 \ \ 10 \ \ 14 \ \ 18 \ \ \rightarrow
\end{align*}
\item If $\Z$ consists of 7 (i.e., $a = 7$) double points then the Hilbert function of $\Z$ is at least
\begin{align*}
1 \ \ 3 \ \ 6 \ \ 10 \ \ 15 \ \ 19 \ \ \cdots
\end{align*}
\item If $\Z$ consists of 8 (i.e., $a = 8$) double points then the Hilbert function of $\Z$ is at least
\begin{align*}
1 \ \ 3 \ \ 6 \ \ 10 \ \ 15 \ \ 20 \ \ \cdots
\end{align*}
\item If $\Z$ consists of 9 (i.e., $a = 9$) double points then the Hilbert functions of $\Z$ is at least
\begin{align*}
1 \ \ 3 \ \ 6 \ \ 10 \ \ 15 \ \ 20 \ \ 24 \ \ 27 \rightarrow
\end{align*}
\end{enumerate}
It can be checked using (\ref{eq3}) that for $6 \le a \le 8$, $h^0(\I_\Y(6)) \le 5$. This implies that, for $6 \le a \le 8$, the space of sextics containing $\Y$ (i.e., the $S_6$'s above) have dimension at most 5, whence $\dim_k \big(I_{2\X}\big)_7 \le 5.$ For the remaining value $a = 9$, since $H_\Z(6) \ge 24$, the space of sextics containing $\Z$ has dimension at most 4. It then follows that the space of sextics containing $\Y$ has dimension at most 4, whence $\dim_k \big(I_{2\X}\big)_7 \le 4.$

The case when $\big(I_{2\X}\big)_7$ has a fixed component is completed.

\medskip\noindent{\bf Let us now assume that $\big(I_{2\X}\big)_7$ is composed with a pencil.}

In this case, since $7$ is a prime number, there exist two linear forms $W$ and $V$ such that any element $H_7$ of $\big(I_{2\X}\big)_7$ is a degree 7 polynomial in $W$ and $V$. This implies that $H_7$ factors into 7 linear factors in $W$ and $V$ (which are also linear in $x,y$ and $z$).

Recall that $F_6$ is the unique generator in degree 6 of $I_{2\X}$, and notice again that $F_6x, F_6y, F_6z \in \big(I_{2\X}\big)_7$. Thus, $F_6x, F_6y$ and $F_6z$ can be written as polynomials of degree 7 in $W$ and $V$. This, in particular, implies that
$$\dfrac{x}{z} = \dfrac{F_6x}{F_6z}, \dfrac{y}{z} = \dfrac{F_6y}{F_6z} \in k\Big(\dfrac{W}{V}\Big).$$
That is,
$$k\Big(\dfrac{x}{z}, \dfrac{y}{z}\Big) \subseteq k\Big(\dfrac{W}{V}\Big).$$
This is a contradiction. Hence, $\big(I_{2\X}\big)_7$ cannot be composed with a pencil. The proposition is proved.
\end{proof}

\begin{proposition} \label{pro.atleast2}
$H_{2\X}(6) \ge 26.$
\end{proposition}

\begin{proof} Suppose, by contradiction that $H_{2\X}(6) < 26$. By Proposition \ref{pro.atmostt} and the fact that $\deg 2\X = 33$, the difference function of the Hilbert function of $2\X$ is of the form
$$\Delta H_{2\X} : \ \ 1 \ \ 2 \ \ 3 \ \ 4 \ \ 5 \ \ 6 \ \ a \ \ b \ \ c \ \ d \ \ 0 \ \ \rightarrow$$
where $a+b+c+d = 12$ and $a \ge b \ge c \ge d \ge 0$. By the assumption that $H_{2\X}(6) < 26$, we also have that $a < 5$. The possibilities for $(a,b,c,d)$ are: $(4,4,4,0)$, $(4,4,3,1)$, $(4,4,2,2)$, $(4,3,3,2)$ and $(3,3,3,3)$.

\noindent {\bf Case 1:} $(a,b,c,d) = (4,4,4,0), (4,4,3,1)$ or $(4,4,2,2)$. By \cite[Proposition 5.2]{GMR}, there is a quartic, $Q_4$, containing a subscheme $\Y$ of $2\X$ of degree 30. Since each simple point on $Q_4$ contributes 2 and each double point contributes 3 to the degree, the only possibilities for $\Y$ are that it consists of 10 double points of $2\X$, or it consists of 8 double points of $2\X$ and 3 simple points of $\X$.

An irreducible quartic has at most 3 singular points, so $Q_4$ is not irreducible. It can also be easily seen that a reduced quartic has at most 6 singular points. Therefore, $Q_4$ is non-reduced. The possibilities are: $Q_4 = C_2^2$, $Q_4 = C_2L^2$, $Q_4 = L^3L_1$, and $Q_4 = L^4$, where $L$ and $L_1$ denote distinct linear forms, and $C_2$ denotes an irreducible conic.

If $Q_4 = C_2^2$ then the conic $C_2$ contains $\ge$ 10 points of $\X$, a contradiction.

If $Q_4 = C_2L^2$ then, since $C_2$ has no singular points, $\ge$ 8 points of $\X$ must be on $L$, a contradiction.

If $Q_4 = L^3L_1$ or $Q_4 = L^4$, then we are again forced to have $L$ contain $\ge$ 8 points of $\X$, a contradiction.

\noindent{\bf Case 2:} $(a,b,c,d) = (4,3,3,2)$. Again, by \cite[Proposition 5.2]{GMR}, there is a cubic $C_3$ containing a subscheme $\Y$ of $2\X$ of degree 26. The only possibility to get multiplicity 26 is that $\Y$ consists of 8 double points and a simple point of $2\X$.

A reduced cubic has at most 3 double points. Thus, $C_3$ is non-reduced. The possibilities for $C_3$ are: $C_3 = L^2L_1$ or $C_3 = L^3$. In both cases the line $L$ is forced to contain at least 8 points of $\X$, a contradiction.

\noindent{\bf Case 3:} $(a,b,c,d) = (3,3,3,3)$. As before, by \cite[Proposition 5.2]{GMR}, there is a cubic $C_3$ containing a subscheme $\Y$ of $2\X$ having  degree 27. This is only possible if $\Y$ consists of 9 double points of $2\X$.

An argument similar to that of Case 2 shows that $C_3$ is non-reduced.  Also in this case we are forced to conclude that there is a line containing at least 9 points of $\X$, a contradiction.

This completes the proof of our proposition.
\end{proof}

\begin{proposition} \label{pro.62}
If $H_{2\X}(6) = 26$ then $H_{2\X}(7) \ge 31$ and thus $H_{2\X} \ge H_{2C_{5,1}}$.
\end{proposition}

\begin{proof} Let $S_6$ and $S_6'$ be any two elements that generate the 2-dimensional vector space $\big(I_{2\X}\big)_6$. By B\'ezout's Theorem $S_6$ and $S_6'$ should meet in a divisor of degree 36, but the 11 double points in $2\X$ already give at least degree 44 for this divisor. Thus, $S_6$ and $S_6'$ share a common factor. In particular, this implies that $\big(I_{2\X}\big)_6$ has a fixed component. We shall call this fixed component $F$.
Our proof proceeds by a case by case analysis  based on the degree of $F$.

\noindent{\it Claim 1.} If $\deg F = 5$ then $\X$ is the same as a $C_{5,1}$ configuration.

\noindent{\it Proof of the Claim.} In this case, elements of $\big(I_{2\X}\big)_6$ are linear combinations of $FL_1$ and $FL_2$, where $L_1$ and $L_2$ are distinct linear forms. Since $L_1$ and $L_2$ intersect in 1 point, $F$ must contain at least 10 double points of $2\X$. By using B\'ezout's Theorem and considering possible factorizations of $F$, it can be shown that this is the case only if $F$ is the product of 5 distinct lines, and the 10 double points on $F$ are the intersection points of these lines. The last double point of $2\X$ must be the intersection of $L_1, L_2$ and $F$. Thus, $\X$ is exactly the configuration $C_{5,1}$, and we are done. \qed

\noindent{\it Claim 2.} $\deg F \not= 4$.

\noindent{\it Proof of the Claim.} Suppose $\deg F = 4$. In this case, a general element of $\big(I_{2\X}\big)_6$ is of the form $FC$, where $C$ is a reduced conic (by Lemma \ref{BertiniType}).

\noindent{\it Case a:} $F$ is reduced. The possibilities for $F$ are: $F$ is irreducible, $F = LC_3$, $F = C_2C_2'$, $F = C_2L_1L_2$, and $F = L_1L_2L_3L_4$, where $L$ and the $L_i$ denote distinct linear forms, $C_2$ and $C_2'$ denote distinct irreducible quadrics, and $C_3$ denotes an irreducible cubic form.

If $F$ is irreducible then $F$ has at most 3 singular points. Thus, the other 8 double points of $2\X$ must be from $C$ or the intersection $C \cap F$. This implies that $C$ contains at least 8 points of $\X$. The Hilbert function of these 8 points must be $1 \ \ 3 \ \ 5 \ \ 7 \ \ 8 \ \ \rightarrow$. This implies that there is a unique such conic $C$. In particular, the pencil $\big(I_{2\X}\big)_6$ has dimension 1, a contradiction.

If $F = LC_3$ then since $C_3$ has at most 1 singular point and $C_3$ intersects $L$ at 3 points, at least 7 double points of $2\X$ come from $C$ or the intersection $C \cap F$. That is, $C$ goes through 7 points of $\X$. The Hilbert function of these 7 points is $1 \ \ 3 \ \ 5 \ \ (6 \text{ or } 7) \ \ 7 \ \ \rightarrow$. Again, in this case, the pencil $\big(I_{2\X}\big)_6$ has dimension 1, a contradiction.

If $F = C_2C_2'$ then since $C_2$ and $C_2'$ intersect at 4 points, again, $C$ must contain at least 7 points of $\X$. As before, this leads to a contradiction.

If $F = C_2L_1L_2$ then since $L_1 \cap L_2$ has 1 point and $C_2 \cap L_i$ has 2 points, $C$ must contain at least 6 points of $\X$. If 4 or 5 of these points are on a line, say $L$, then $L$ intersects an element of $\big(I_{2\X}\big)_6$ at 4 or 5 double points (of $2\X$). A line and a degree 6 form should only meet at 6 points, and thus, $L$ is a component of $\big(I_{2\X}\big)_6$. We are back in Case 1. If at most 3 points of these 6 points are on a line, then the Hilbert function of these 6 points is $1 \ \ 3 \ \ 5 \ \ 6 \ \ \rightarrow$. Once again, this leads to a contradiction of the fact that $\big(I_{2\X}\big)_6$ is a pencil.

If $F = L_1L_2L_3L_4$ then, since these lines intersect in at most 6 points, $C$ must contain at least 5 fixed points of $\X$. If 4 of these 5 points are on a line, say $L$, then (as before) $L$ must be a component of $\big(I_{2\X}\big)_6$, and we are back to Case 1. If at most 3 of these 5 points are on a line, then the Hilbert function of the 5 points is $1 \ \ 3 \ \ 5 \ \ \rightarrow$, and, as before, we are lead to a contradiction.

\noindent{\it Case b:} $F$ is non-reduced. The possibilities for $F$ are: $F = C_2^2$, $F = C_2L^2$, $F = L^3L_1$, $F = L^2L_1L_2$ and $F = L^4$, where $L$ and $L_i$ denote linear forms and $C_2$ denotes an irreducible conic.

If $F = C_2^2$ then, since $C$ has at most 1 singular point, $C_2$ must contain at least 10 points of $\X$, a contradiction to the fact that $\X$ has generic Hilbert function.

If $F = C_2L^2$ then, since $C$ and $C_2$ intersect at 4 points, $C_2$ has no singular points and $C$ has at most 1 singular point, $L$ must contain at least 6 points of $X$. Again, this is a contradiction.

If $F = L^3L_1$, $F = L^2L_1L_2$ or $F=L^4$, then it can be shown that too many points of $X$ are on a line. \qed

\noindent{Claim 3.} $\deg F \not= 3.$

\noindent{\it Proof of the Claim.} Suppose $\deg F = 3$. In this case, a general element of $\big(I_{2\X}\big)_6$ has the form $FC_3$, where $C_3$ is a reduced cubic (by Lemma \ref{BertiniType}).

\noindent{\it Case a:} $F$ is reduced. The possibilities for $F$ are: $F$ is irreducible, $F = C_2L_1$, $F = L_1L_2L_3$, where the $L_i$ are distinct linear forms and $C_2$ denotes an irreducible conic.

If $F$ is irreducible then $F$ has at most 1 singular point. Thus, $C_3$ must contain at least 10 fixed points of $\X$. Since $\X$ has generic Hilbert function (and a point imposes at most one independent condition to cubics), the Hilbert function of these 10 points must be $1 \ \ 3 \ \ 6 \ \ (\ge 9) \ \ \cdots$ This implies that there is a unique such cubic $C_3$. This leads to a contradiction of the fact that $\big(I_{2\X}\big)_6$ is a pencil.

Suppose that $F = C_2L_1$. Since $C_2 \cap L_1$ has 2 points, $C_3$ must contain $a \ge 0$ double points of $2\X$ and at least $(9-a)$ intersection points $F \cap C_3$. Observe that the two cubics $C_3$ and $C_3'$ (for which $FC_3$ and $FC_3'$ generate $\big(I_{2\X}\big)_6$) intersect at 9 points, but $a$ double points and $(9-a)$ single points give $4a+(9-a) = 9+3a$ points of intersection. Thus, if $a \ge 1$, then $C_3$ and $C_3'$ share a common factor that could be absorbed into the fixed component $F$ of $\big(I_{2\X}\big)_6$, and we are back to a previous case. Assume that $a = 0$ and $C_3$ contains the 9 intersection points of $F \cap C_3$ (a complete intersection of two cubics). The Hilbert function of these 9 points is $1 \ \ 3 \ \ 6 \ \ 8 \ \ 9 \ \ \rightarrow$. Therefore, there are exactly two cubics containing them (the 9 points). In this case, there is only one such $C_3$ (the other cubic is $F$), and this leads to a contradiction of the fact that $\big(I_{2\X}\big)_6$ is a pencil.

Suppose that $F = L_1L_2L_3$. Since the lines $L_i$ intersect at 3 points, $C_3$ must contain $a \ge 0$ double points of $2\X$ and at least $(8-a)$ points of the intersection $F \cap C_3$. As before, the two cubics $C_3$ and $C_3'$ (for which $FC_3$ and $FC_3'$ generate $\big(I_{2\X}\big)_6$) intersect at 9 points, so if $a \ge 1$ then $a$ double points and $(8-a)$ single points would give $4a+(8-a) = 8+3a > 9$ intersection points, and therefore, we can enlarge the fixed component $F$ and are back to a previous case. Assume that $a = 0$. Then $C_3$ contains at least 8 intersection points of $F \cap C_3$. In this case, $C_3$ will intersect two of the lines $L_1, L_2$ and $L_3$ in at least 6 points. The union of these two lines (a conic) now contains at least 9 points of $\X$, a contradiction.


\noindent{\it Case b:} $F$ is non-reduced. The possibilities for $F$ are: $F = L^2L_1$ and $F = L^3$, where $L$ and $L_1$ denote distinct linear forms.

Suppose that $F = L^2L_1$. Since $L$ contains at most 5 points of $\X$, $C_3 \cap L_1$ has 3 points and $C_3$ has at most 3 singular points, it must be the case that $L$ contains exactly 5 points of $\X$ and $C_3$ has exactly 3 double points of $2\X$. The Hilbert function of these 3 double points is $1 \ \ 3 \ \ 6 \ \ 9 \ \ 12 \ \ \rightarrow$. Thus, there is at most one such cubic $C_3$. That is, $\big(I_{2\X}\big)_6$ has dimension at most 1, a contradiction.

Suppose that $F = L^3$. Since $C_3$ has at most 3 singular points, $L$ must contain at least 8 points of $\X$. This is also a contradiction. \qed

\noindent{\it Claim 4.} $\deg F \not= 2$.

\noindent{\it Proof of the Claim.} Suppose $\deg F = 2$. A general element of $\big(I_{2\X}\big)_6$ has the form $FQ_4$, where $Q_4$ is a reduced quartic.

\noindent{\it Case a:} $F$ is irreducible. In this case $F \cap Q_4$ has 8 points. Thus, $Q_4$ must contain at least 3 double points of $2\X$. Suppose the subscheme $\Z$ of $2\X$ determining the space of such $Q_4$'s consists of $a \ge 3$ double points and $(11-a)$ points of the intersection $F \cap Q_4$. Observe that two quartics $Q_4$ and $Q_4'$ (for which $FQ_4$ and $FQ_4'$ generate $\big(I_{2\X}\big)_6$) intersect at 16 points, but $a$ double points and $(11-a)$ single points give $4a + (11-a) = 11+3a > 16$ points of intersection for any $a \ge 3$. Thus, $Q_4$ and $Q_4'$ share a common factor, which now can be absorbed into the fixed component $F$ of $\big(I_{2\X}\big)_6$, and we are back in previous cases.


\noindent{\it Case b:} $F$ is reducible. The possibilities are: $F = L_1L_2$ and $F = L^2$, where $L$ and the $L_i$'s denote distinct linear forms.

Consider the case where $F = L_1L_2$. Note that the intersection point, $L_1 \cap L_2$, and the 8 intersection points, $F \cap Q_4$, cannot all be in $\X$ (otherwise, $F$ is a conic containing 9 points of $\X$). Thus, $Q_4$ must contain at least 3 double points in $2\X$, and the subscheme $\Z$ of $2\X$ determining such $Q_4$'s consists of $a \ge 3$ double points and at least $(10-a)$ points of the intersection $F \cap Q_4$. By an argument similar to one we've seen before, we can enlarge the fixed component $F$ and appeal to previous cases.

Consider the case where $F = L^2$. Since $L$ contains at most 5 points of $\X$, $Q_4$ must contain at least 6 double points in $2\X$. These 6 double points give at least 24 intersection points of two such quartics $Q_4$ and $Q_4'$. Again, this leads to the situation where we can enlarge $F$ and appeal to previous cases. \qed

\noindent{\it Claim 5.} $\deg F \not= 1$.

\noindent{\it Proof of the Claim.} Suppose $\deg F = 1$. A general element of $\big(I_{2\X}\big)_6$ has the form $F.Q_5$, where $Q_5$ is a reduced quintic.

Since the line described by $F$ contains at most 5 points of $\X$, the quintic $Q_5$ must have at least 6 double points of $2\X$ (and passes through the points of $\X$ on $F$). Observe that the two distinct quintics in the pencil $\big(I_{2\X}\big)_6$ meet at 25 points. However, the 6 double points and the 5 simple points on these quintics give at least 29 points of intersection. This is the case only if the two quintics share a common component. In particular, we can again enlarge the fixed component $F$ and appeal to previous cases.

This concludes the claim and proves the proposition.
\end{proof}



\begin{thebibliography}{100}
\bibitem{AH0} Alexander, J.; Hirschowitz, A. {\it The blown-up Horace method: application to fourth-order interpolation}. Invent. Math. {\bf 107} (1992), no. 3, 585–602.
\bibitem{AH} Alexander, J.; Hirschowitz, A. {\it Polynomial interpolation in several variables}. J. Algebraic Geom. {\bf 4} (1995), no. 2, 201-222.
\bibitem{BC} Bocci, C.; Chiantini, L. {\it The effect of points fattening on postulation}. J. Pure Appl. Algebra. {\bf 215} (2011), 89-98.
\bibitem{C} Ciliberto, C. {\it Geometric aspects of polynomial interpolation in more variables and of Waring's problem}. European Congress of Mathematics, Vol. I (Barcelona, 2000), 289-316, Progr. Math., {\bf 201}, Birkh\"auser, Basel, 2001.
\bibitem{CHT} Cooper, S.; Harbourne, B.; Teitler, Z. {\it Combinatorial bounds on Hilbert functions of fat points in projective space.} J. Pure Appl. Algebra. {\bf 215} (2011), 2165-2179.
\bibitem{GGP} Geramita, A.V.; Gimigliano, A.; Pitteloud, Y. {\it Graded Betti numbers of some embedded rational $n$-folds}. Math. Ann. {\bf 301} (1995), 363-380.
\bibitem{GHM1} Geramita, A. V.; Harbourne, Brian; Migliore, Juan. {\it Classifying Hilbert functions of fat point subschemes in $\PP^2$}. Collect. Math. {\bf 60} (2009), no. 2, 159-192.
\bibitem{GHM2} Geramita, A. V.; Harbourne, Brian; Migliore, Juan. {\it Hilbert functions of fat point subschemes of the plane: the two-fold way.} {\tt arXiv: 1101.5140.}
\bibitem{GMR} Geramita, A. V.; Maroscia, P.; Roberts, L. G. {\it The Hilbert function of a reduced $k$-algebra}. J. London Math. Soc. (2) {\bf 28} (1983), no. 3, 443-452.
\bibitem{GMS} Geramita, A. V.; Migliore, J.; Sabourin, L. {\it On the first infinitesimal neighborhood of a linear configuration of points in $\PP^2$}. J. Algebra {\bf 298} (2006), no. 2, 563-611.
\bibitem{Ger} Geramita, A.V. {\it Inverse systems of fat points: Waring's problem, secant varieties of Veronese varieties and parameter spaces for Gorenstein ideals.} In The Curves Seminar at Queen's, Vol. X (Kingston, ON, 1995), volume 102 of Queen's Papers in Pure and Appl. Math. {\bf 114}. Queen's Univ., Kingston, ON, 1996.
\bibitem{I} Iarrobino, A. {\it Inverse system of a symbolic power. II. The Waring problem for forms.} J. Algebra {\bf 174} (1995), no. 3, 1091-1110.
\bibitem{I-K} Iarrobino, A.; Kanev, V. Power sums, Gorenstein algebras, and determinantal loci. Volume 1721 of Lecture Notes in Mathematics. Springer-Verlag, Berlin, 1999.
\end{thebibliography}
\end{document}